\documentclass{amsart}
\usepackage[all]{xy}
\usepackage{color}
\usepackage{amsthm}

\usepackage[colorlinks=true,hypertex]{hyperref}

\numberwithin{equation}{subsection}

\swapnumbers

\newtheorem{theorem}[equation]{Theorem}
\newtheorem*{theorem*}{Theorem}
\newtheorem{lemma}[equation]{Lemma}

\newtheorem*{conjecture*}{Mamma Conjecture}
\newtheorem*{conjecture1*}{Mamma Conjecture (revisited)}
\newtheorem{proposition}[equation]{Proposition}
\newtheorem{corollary}[equation]{Corollary}
\newtheorem*{corollary*}{Corollary}

\theoremstyle{remark}
\newtheorem{definition}[equation]{Definition}
\newtheorem{definitions}[equation]{Definitions}

\newtheorem{examples}[equation]{Examples}

\newtheorem{notation}[equation]{Notation}

\theoremstyle{remark}
\newtheorem{remark}[equation]{Remark}

\newcommand{\ko}{\: , \;}
\newcommand{\cA}{{\mathcal A}}
\newcommand{\cB}{{\mathcal B}}
\newcommand{\cC}{{\mathcal C}}
\newcommand{\cD}{{\mathcal D}}
\newcommand{\cE}{{\mathcal E}}
\newcommand{\cF}{{\mathcal F}}

\newcommand{\cM}{{\mathcal M}}

\newcommand{\cS}{{\mathcal S}}

\newcommand{\cU}{{\mathcal U}}
\newcommand{\cV}{{\mathcal V}}

\newcommand{\bbC}{\mathbb{C}}
\newcommand{\bbD}{\mathbb{D}}

\newcommand{\bfE}{\mathbf{E}}

\newcommand{\bbK}{\mathbb{K}}
\newcommand{\bbL}{\mathbb{L}}
\newcommand{\bbR}{\mathbb{R}}
\newcommand{\bbT}{\mathbb{T}}
\newcommand{\bbQ}{\mathbb{Q}}
\newcommand{\bbZ}{\mathbb{Z}}

\DeclareMathOperator{\hocolim}{hocolim}

\DeclareMathOperator{\id}{id}
\DeclareMathOperator{\incl}{incl}
\DeclareMathOperator{\Or}{Or}

\DeclareMathOperator{\Ind}{Ind}

\DeclareMathOperator{\Mod}{Mod}
\DeclareMathOperator{\Mot}{Mot}

\DeclareMathOperator{\K}{K}
\DeclareMathOperator{\Ch}{Ch}

\newcommand{\KB}[1]{K(#1;\cB)}
\newcommand{\dgcat}{\mathsf{dgcat}}
\newcommand{\perf}{\mathsf{perf}_{\dg}}
\newcommand{\cat}{\mathsf{cat}}
\newcommand{\dgcatf}{\dgcat_{\mathsf{f}}}
\newcommand{\dgcatsf}{\dgcat_{\mathsf{sf}}}
\newcommand{\dg}{\mathsf{dg}}

\newcommand{\spe}{\mathsf{Spt}}
\newcommand{\Spe}{\mathsf{Spt}}
\newcommand{\uHom}{\underline{\mathsf{Hom}}}
\newcommand{\Hom}{\mathsf{Hom}}
\newcommand{\HomC}{\uHom_{\,!}}
\newcommand{\HomA}{\uHom_{\,\mathsf{add}}}
\newcommand{\Locuns}{\Loc_{\widetilde{\cE^s_{un}},\,p,\,\Sigma}}
\newcommand{\rep}{\mathsf{rep}_{\dg}}

\newcommand{\Cat}{\mathsf{Cat}}
\newcommand{\CAT}{\mathsf{CAT}}
\newcommand{\Fun}{\mathsf{Fun}}
\newcommand{\dgHo}{\mathsf{H}^0}
\newcommand{\dgZo}{\mathsf{Z}^0}
\newcommand{\Ho}{\mathsf{Ho}}
\newcommand{\HO}{\mathsf{HO}}

\newcommand{\op}{^{\mathsf{op}}}
\newcommand{\sSet}{\mathsf{sSet}}
\newcommand{\sSetp}{\mathsf{sSet}_{\bullet}}
\newcommand{\Map}{\mathsf{Map}}
\newcommand{\uMap}{\underline{\mathsf{Map}}}
\newcommand{\St}{\mathsf{St}}
\newcommand{\stab}{\mathsf{stab}}
\newcommand{\Loc}{\mathsf{L}}
\newcommand{\too}{\longrightarrow}
\newcommand{\isoto}{\stackrel{\sim}{\to}}
\newcommand{\isotoo}{\stackrel{\sim}{\longrightarrow}}
\newcommand{\Left}{\mathbb{L}}
\newcommand{\Right}{\mathbb{R}}
\newcommand{\LInd}{\Left\!\Ind}
\newcommand{\Efund}{\bfE^{\mathsf{fund}}}
\newcommand{\dgS}{\cS}
\newcommand{\dgD}{\cD}

\newcommand{\ie}{\textsl{i.e.}\ }
\newcommand{\eg}{\textsl{e.g.}}

\newcommand{\dgcell}{strictly finite dg cell}
\newcommand{\dgcells}{strictly finite dg cells}

\newcommand{\Madd}{\Mot^{\mathsf{add}}_{\dg}}
\newcommand{\cMadd}{\cM ot^{\,\mathsf{add}}_{\dg}}
\newcommand{\Uadd}{\cU^{\,\mathsf{add}}_{\,\dg}}

\begin{document}

\title[The fundamental isomorphism conjecture]{The fundamental isomorphism conjecture\\via non-commutative motives}
\author{Paul~Balmer and Gon{\c c}alo~Tabuada}

\address{Paul Balmer, Department of Mathematics, UCLA, Los Angeles, CA 90095-1555, USA}
\email{balmer@math.ucla.edu}
\urladdr{http://www.math.ucla.edu/$\sim$balmer}

\address{Gon{\c c}alo Tabuada, Department of Mathematics, MIT, Cambridge, MA 02139, USA}
\email{tabuada@math.mit.edu}
\urladdr{http://www.math.mit.edu/$\sim$tabuada}

\date{\today}


\thanks{The first author was supported by NSF grant {\tt DMS-0969644} and the second author by the Clay Mathematics Institute and by the FCT-Portugal grant {\tt PTDC/MAT/098317/2008}.}

\begin{abstract}
Given a group, we construct a fundamental additive functor on its orbit
category. We prove that any isomorphism conjecture valid
for this fundamental additive functor holds for all additive functors, like $K$-theory, cyclic homology, topological Hochschild homology, etc. Finally, we reduce
this fundamental isomorphism conjecture to $K$-theoretic ones.
\end{abstract}

\maketitle

\vskip-\baselineskip
\vskip-\baselineskip
\section{Introduction and statement of results}
\subsection{Isomorphism conjectures}\label{sub:isomconj}
The Farrell-Jones \emph{isomorphism conjectures} are important driving
forces in current mathematical research and imply well-known conjectures due to Bass, Borel, Kaplansky, Novikov; see a survey in L\"uck~\cite{ICMLuck}.

Given a group $G$, the Farrell-Jones conjectures predict
the value of algebraic $K$- and $L$-theory of the group ring $RG$
in terms of their values on the virtually cyclic subgroups of~$G$\,; here $R$ is a fixed base commutative ring. In~\cite{DL}, Davis and L{\"u}ck proposed
the following unified setting for these isomorphism
conjectures\,; see~\S\,\ref{se:DL}. Let $\cF$ be a family of subgroups of~$G$ and
$\bfE:\Or(G) \rightarrow \spe$ a functor from the orbit category
of~$G$ to spectra. The {\em
$(\bfE,\cF,G)$-assembly map} is the induced map
\begin{equation}\label{eq:ass}
\underset{\Or(G,\cF)}\hocolim\, \bfE \longrightarrow
\underset{\Or(G)}\hocolim\, \bfE = \bfE(G)\,,
\end{equation}
where $\Or(G,\cF)\subset\Or(G)$ is the orbit category restricted on~$\cF$.
We say that \emph{the functor $\bfE$ has the $\cF$-assembly property
for~$G$} when the map~\eqref{eq:ass} is a stable weak equivalence, \ie when it
induces an isomorphism on stable homotopy groups. When we speak of \emph{the
$(\bfE,\cF,G)$-isomorphism conjecture}, we refer to the expressed
hope that this property holds for a particular choice of $\bfE$,
$\cF$ and~$G$.
Davis and L{\"u}ck proved (see also \cite{HP} for details on the proof) that the Farrell-Jones
conjecture in $K$-theory for~$G$ is equivalent to the
$(\bbK,\cV C,G)$-isomorphism conjecture, where $\bbK$ is non-connective $K$-theory (see \S\,\ref{sub:nonconn}) and $\cV C$ the family of virtually cyclic subgroups of~$G$; and similarly for $L$-theory. The first step in their approach is the construction of a functor to $R$-linear categories
\begin{equation}\label{eq:comp}
\Or(G) \overset{\overline{?}}{\too} \mathsf{Grp} \overset{R[-]}{\too} R\text{-}\cat\,,
\end{equation}
composed of the transport groupoid functor and the $R$-linearization functor\,; see~\S\,\ref{se:DL}.

In addition, the literature contains many
variations on the above theme, replacing the $K$- and $L$-theory functors by other
functors~$\bfE$, and the category of spectra by other model categories~$\cM$. See for instance the isomorphism conjecture for homotopy $K$-theory ($KH$)~\cite[\S\,7]{BL}, for Hochschild homology ($HH$) and cyclic homology ($HC$)~\cite[\S\,1]{LR}, or for topological Hochschild homology ($THH$)~\cite[\S\,6]{Luck}. This simple idea of letting the functor $\bfE$ and the category $\cM$ float freely generates a profusion of potential isomorphism conjectures\,:
\begin{equation}\label{eq:skein}
\vcenter{
\xymatrix@C=8em@R=0.4em{
&& \spe
\\
&& \spe \kern-1em
\\
\Or(G) \ar[rruu]^(.7){\bbK}
 \ar[rru]^(.75){KH} \ar[rr]^(.8){HH}
 \ar[rrd]^(.75){HC} \ar[rrdd]^(.75){THH}  \ar[rrddd]^(.7){\bfE}_(.55){{\vdots}}
&& \Ch(R) \kern-2em
\\
&& \Ch(R) \kern-1em
\\
&& \spe\\
&& \cM \,,
}}
\end{equation}
where $\Ch(R)$ stands for the category of complexes of $R$-modules.
Each of these isomorphism conjectures has already been proved for large classes of groups using a variety of different methods.
Our goal in this article is not to prove any of these conjectures for any class of groups. We are rather interested
in the general organization and deeper properties behind this somewhat exuberant herd of conjectures. Our guiding questions are the following\,:

\medbreak
\noindent\textbf{Question~A\,: }\textit{Is there a fundamental isomorphism conjecture implying all others\,?}
\medbreak
\noindent\textbf{Question~B\,: }\textit{If this is the case, can this fundamental isomorphism conjecture be described solely in terms of classical invariants\,?}
\medbreak

We provide positive answers to those questions. As we shall see, the isomorphism conjecture for $K$-theory will play a central role, probably confirming the secret feelings of some experts. The precise formulation of our answers uses the theory of {\em non-commutative motives} initiated in~\cite{Thesis}. Connecting those techniques to the world of isomorphism conjectures is the technical bulk of the article. We believe that this bridge will encourage the adoption of these new techniques by researchers working on isomorphism conjectures.
\subsection{Non-commutative motives}\label{sub:NCM}
A {\em differential graded (=dg) category}, over our fixed base commutative ring $R$, is a category enriched over
cochain complexes of $R$-modules (morphisms sets are complexes)
in such a way that composition fulfills the Leibniz rule\,:
$d(f\circ g)=d(f)\circ g+(-1)^{\textrm{deg}(f)}f\circ(dg)$; see Keller~\cite{ICM} and \S\,\ref{se:dgcat} for further explanations about dg categories. There is a Quillen model structure on $\dgcat$, the category of small dg categories, with weak equivalences being derived Morita equivalences (see \S\,\ref{sub:Morita}).

Many of the classical invariants such as Hochschild and cyclic homology,
connective, non-connective, and homotopy $K$-theory, and even topological Hochschild homology, extend naturally from $R$-algebras to dg categories\,; see \S\,\ref{se:add}. In order to study all these invariants simultaneously the notion of {\em additive invariant} was introduced in~\cite[\S\,15]{Additive}. This theory makes use of the language of Grothendieck derivators, a formalism which
allows us to state and prove precise universal properties\,; see Appendix~\ref{appendix}.
Let $\mathit{E}: \HO(\dgcat) \to \bbD$ be a morphism of derivators, from the
derivator associated to $\dgcat$, to
a strong triangulated derivator~$\bbD$.
We say that $\mathit{E}$ is an {\em additive invariant} if it preserves filtered homotopy
colimits and the terminal object, and if it sends split exact sequences to direct sums
$$
\begin{array}{lcccr}
\xymatrix{
 \cA \ar[r]_{I} & \cB \ar@<-1ex>[l]_{T}
\ar[r]_P & \cC \ar@<-1ex>[l]_{S}
} & &\longmapsto & & [\mathit{E}(I)\,\, \mathit{E}(S)]:\ \mathit{E}(\cA)\oplus \mathit{E}(\cC) \stackrel{\sim}{\too} \mathit{E}(\cB)\,.
\end{array}
$$
By the additivity results of Keller~\cite{cyclic}, Waldhausen~\cite{Wald}, Schlichting~\cite{Marco}, Weibel~\cite{Weibel}, and Blumberg-Mandell~\cite{BlumbergMandell} (see also~\cite{MacLane}), all the above classical theories are additive invariants (see \S\,\ref{se:add}).
In~\cite[Def.~15.1]{Additive} the {\em universal additive invariant} was constructed
\begin{equation*}
 \Uadd: \HO(\dgcat) \too \Madd\,.
\end{equation*}
It is universal in the following sense. Given any strong triangulated derivator $\bbD$ we have an equivalence of categories
\begin{equation}\label{eq:cat}
(\Uadd)^{\ast}:\ \HomC\big(\Madd, \bbD\big) \stackrel{\sim}{\too} \HomA\big(\HO(\dgcat), \bbD\big)\,,
\end{equation}
where the left-hand side denotes the category of homotopy colimit preserving morphisms of derivators and the right-hand side denotes the category of additive invariants\,; see~\cite[Thm.~15.4]{Additive}. In words, this means that every additive invariant on dg categories (an object in the right-hand category) factors essentially uniquely via $\Uadd$, that is, via the universal additive derivator~$\Madd$. 

Furthermore, this universal additive derivator $\Madd$ admits an explicit Quillen model $\cMadd$; see \S\,\ref{se:Madd}.
Because of its universal property, which is reminiscent of the theory of motives, the derivator $\Madd$ is called the {\em additive motivator}, and its base category $\Madd(e)$, which is the homotopy category $\Ho(\cMadd)$, is called the {\em triangulated category of non-commutative motives}.
\subsection{Fundamental isomorphism conjecture}
The above notion of additivity combined with functor \eqref{eq:comp} yields the notion of an {\em additive functor on the orbit category}; see Definition~\ref{addfunctor}. In particular, all functors mentioned in diagram \eqref{eq:skein} are additive on~$\Or(G)$.
\begin{remark}[Limitations]
We would like to mention that invariants involving \emph{enriched} structures, like topological $K$-theory or $L$-theory are not additive invariants in this first, elementary sense. Including the Baum-Connes conjecture in our treatment would require the definition of topological $K$-theory of the reduced $C^*$-algebra via dg categories and this does not exist at the moment. Similarly, including $L$-theory would require to carry over dualities throughout the game and this is another story. 
\end{remark}
An application of universality~\eqref{eq:cat} is the following\,:
\begin{theorem}[see Thm.\,\ref{thm:factor}]\label{thm:Imain1}
Let $G$ be a group and let $R$ be a commutative ring. Then there exists a fundamental additive functor on its orbit category
\begin{equation*}
\Efund:\ \Or(G) \stackrel{\overline{?}}{\too} \mathsf{Grp} \stackrel{R[-]}{\too} R\text{-}\cat \subset \dgcat \stackrel{\Uadd}{\too} \cMadd
\end{equation*}
through which all additive functors on $\Or(G)$ factor.
\end{theorem}
Intuitively Theorem~\ref{thm:Imain1} allows us to comb
the skein~\eqref{eq:skein} from the left to isolate a fundamental
additive functor
\begin{equation*}
\vcenter{
\xymatrix@C=8em@R=0.4em{
&& \spe
\\
&& \spe \kern-1em
\\
\Or(G) \ar[r]^-{\displaystyle\Efund}
& \cMadd
 \ar[ruu]^(.7){\overline{\bbK}}
 \ar[ru]^(.75){\overline{KH}} \ar[r]^(.8){\overline{HH}}
 \ar[rd]^(.75){\overline{HC}} \ar[rdd]^(.75){\overline{THH}} \ar[rddd]^(.75){\overline{\bfE}}_-{\vdots}
& \Ch(R) \kern-2em
\\
&& \Ch(R) \kern-1em
\\
&& \spe \\
&& \cM
}}\end{equation*}
A key point is that the right-hand
functors $\overline{E}$ preserve homotopy colimits (not only filtered ones). Hence they will preserve \emph{any} assembly property that
$\Efund$ might enjoy. We then obtain the following answer to our Question~A\,:
\begin{corollary}[see Cor.\,\ref{cor:ass}]\label{cor:Imain}
Let $G$ be a group and $\cF$ a family of subgroups. If the
fundamental additive functor $\Efund$ has the $\cF$-assembly
property, so do all additive functors on~$\Or(G)$.
\end{corollary}
\goodbreak
Let us make clear that the $\cF$-assembly property for $\Efund$ has essentially no chance to hold for random choices of~$G$, $\cF$ and~$R$. For instance, if $\cF=\cV C$ this property would imply the $(K, \cV C, G)$-isomorphism conjecture for $R=\bbZ$ and for $K$ being {\em connective} $K$-theory (see \S\,\ref{sub:conn}). And this is known to fail because of the Bass-Heller-Swan decomposition; see~\cite[Rem.~15]{Handbook}. However, if $R$ is a \emph{regular} ring (\ie noetherian and of finite projective dimension) in which the orders of all finite subgroups of $G$ are invertible, then the above obstruction vanishes because the $(K, \cV C, G)$-isomorphism conjecture follows from the Farrell-Jones conjecture\,; see Proposition~\ref{prop:reduction}. This might suggest the following ``mother'' of many  isomorphism conjectures\,:
\begin{conjecture*}
Given a group $G$, the fundamental additive invariant $\Efund$ has the
$\cV C$-assembly property when the base ring $R$ is regular and the orders of all finite subgroups of $G$ are invertible in~$R$. A large class of examples is given by taking $R=\bbQ$ or $\bbC$ and $G$ arbitrary. Another large class of examples is given by taking $R=\bbZ$ and $G$ torsion-free.
\end{conjecture*}

Corollary~\ref{cor:Imain} says that the Mamma conjecture implies {\em all} additive conjectures on the market, for that base ring~$R$ and that group~$G$, with respect to virtually cyclic subgroups. Note that our choice of the family of virtually cyclic groups is merely borrowed from Farrell-Jones and another family $\cF$ might be preferable. In any case, the main result is that once this is achieved for some family $\cF$, then all additive functors will automatically inherit the same $\cF$-assembly property.

\medbreak

Stated differently, instead of multiplying the articles on variations on the isomorphisms conjectures for this and that additive invariant, mathematicians can now try to attack \emph{one single} conjecture and deduce all other ones. It is of course important to know how wild this conjecture is, as asked in our Question~B. We discuss this problem now.

\medbreak

We would like to reduce the $\cF$-assembly
property for $\Efund$, whose importance should now be clear, to the
$\cF$-assembly property for more down-to-earth functors. To do this, we consider functors which are cooked up via
$K$-theory and dg categories as
follows. Given a small dg category $\cB$, consider the functor
$\KB{-}:\Or(G)\to\spe$ defined for every $G/H\in\Or(G)$ by
$$
\KB{G/H}:=K\big(\rep\big(\cB,R[\overline{G/H}]\big)\big)\,.
$$
Some explanations are in order. For any small dg category $\cA$, we denote by $\rep(\cB,\cA)$ the internal Hom-functor, between $\cB$ and $\cA$, in the derived Morita homotopy category; see \S\,\ref{sub:internal}. If $\cB$ is the dg category $\underline{R}$ with one object and with $R$ as dg algebra of endomorphisms, then the functor
$\KB{-}$ reduces to the usual connective $K$-theory
functor~$K$. Hence, when $\cB$ is a general
small dg category, the functor $\KB{-}$ can be thought of as a
``coefficients variant'' of $K$; see Example~\ref{ex:coef}.
The functor $\KB{-}$ is not additive in general,
mainly because $\cB$ might be too large. Therefore, we restrict to dg categories $\cB$ which are \emph{homotopically
finitely presented}; see Definition~\ref{def:h-fin}. Heuristically, this
condition is the homotopical version of the classical notion of finite
presentation. In particular the above example $\cB=\underline{R}$ is homotopically
finitely presented. Our answer to Question~B is\,:
\begin{theorem}\label{thm:Imain}
Let $G$ be a group and $\cF$ be a family of subgroups. Then the
following conditions are equivalent\,:
\begin{enumerate}
\item The fundamental additive functor $\Efund$ has the $\cF$-assembly property for~$G$.
\smallbreak
\item The additive functors $\KB{-}$ have the
$\cF$-assembly property for~$G$, for all homotopically finitely
presented dg categories~$\cB$.
\smallbreak
\item The additive functors $\KB{-}$ have the $\cF$-assembly property for~$G$ for
all \dgcells~$\cB$ (see Definition~\ref{I-cell}).
\end{enumerate}
\end{theorem}
The proof occupies~\S\,\ref{se:proofmain} and is based on the co-representability theorem~\ref{thm:co-representability}. The \dgcells\ of~(3) form a set of homotopically finitely presented
dg categories which are especially small. Roughly speaking, they are
the dg category analogues of finite CW-complexes, namely
they are built by attaching finitely many basic cells, chosen among
the dg analogues $\dgS(n-1)\to\dgD(n)$ of the
topological inclusion $S^{n-1}\hookrightarrow D^n$; see Definition~\ref{I-cell}.

Via Theorem~\ref{thm:Imain}, the Mamma conjecture now boils down to $K$-theory\,:

\begin{conjecture1*}\label{conj:fund1}
Given a group $G$, the functors $\KB{-}$ have the $\cV C$-assembly property for
all \dgcells~$\cB$, when the base ring $R$ is regular and the orders of all finite subgroups of $G$ are invertible in~$R$.
\end{conjecture1*}
When $R$ is a regular ring where the orders of all finite subgroups of $G$ are invertible and $\cB$ is the dg category $\cB=\underline{R}$, the above conjecture basically is the Farrell-Jones conjecture; see Remark~\ref{rk:converse}. Hence, the Mamma conjecture amounts to a coefficients variant of the classical Farrel-Jones conjecture, with strictly finite dg coefficients~$\cB$. Its importance (and that of Theorem~\ref{thm:Imain}) relies on the fact that it simultaneously implies all additive isomorphism conjectures on the market and yet is described solely in terms of $K$-theory. One can therefore expect that future research will adapt existing proofs of the Farrell-Jones conjecture for specific classes of groups to prove the Mamma conjecture, with the benefits explained above.

\medbreak

At some stage, and at least before \S\,\ref{sub:coef}, the reader who is not familiar with the language of Grothendieck derivators should proceed to Appendix~\ref{appendix}, where we also prove that the operations of stabilization and of left Bousfield localization of derivators commute (Theorem~\ref{commut}). The latter result is of independent interest.

\section{The Davis and L\"uck approach}\label{se:DL}

In this section, we recall Davis and L{\"u}ck's reformulation~\cite{DL} of the Farrell-Jones conjecture in $K$-theory. This will be the stepping stone for the construction of the fundamental additive functor in~\S\,\ref{se:univ}. Let $G$ be a (fixed) group.

\subsection{The orbit category}\label{sub:orbit}
The {\em orbit category $\Or(G)$ of $G$} has as objects the homogeneous
$G$-spaces $G/H$, considered as left $G$-sets, and as morphisms the $G$-equivariant maps. A {\em family $\cF$ of subgroups of $G$} is a non-empty set of
subgroups of $G$ which is closed under conjugation and finite intersection. Examples of families of subgroups are given by the family $\cF in$ of finite subgroups, by the family of cyclic subgroups (finite and infinite), and by the family $\cV C$ of
virtually cyclic subgroups\,; recall that $H$ is virtually
cyclic if it contains a cyclic subgroup of finite index.

The {\em orbit category} $\Or(G,\cF)$ {\em restricted on $\cF$} is the full subcategory of $\Or(G)$ consisting of those objects $G/H$ for which $H$ belongs to $\cF$.
\subsection{$\cF$-assembly property}
\label{assemblage}
The $\cF$-assembly property can be generalized from spectra (\S\,\ref{sub:isomconj}) to any target model category~$\cM$. Let $\cF$ be a family of
subgroups of $G$ and let $\bfE: \Or(G) \to \cM$ be a functor. The {\em $(\bfE,\cF,G)$-assembly map} is the map
\begin{equation}\label{eq:Gassemb}
\underset{\Or(G,\cF)}\hocolim\,\, \bfE \too \underset{\Or(G)}\hocolim\,\, \bfE = \bfE(G)
\end{equation}
in $\cM$. We say that $\bfE$ {\em has the $\cF$-assembly property (for~$G$)} when that map is an isomorphism in $\Ho(\cM)$.

A typical approach in the Davis and L\"uck philosophy (mostly with $\cM=\spe$) is the following\,: Given $G$
and $\bfE$, find as small a family $\cF$ as possible for which
$\bfE$ has the $\cF$-assembly property. For instance, for the Farrell-Jones
isomorphism conjectures in $K$- and $L$-theory, one expects $\cF$ to reduce to virtually
cyclic subgroups.

Conceptually, the $\cF$-assembly property for a functor
$\bfE:\Or(G)\to \cM$ essentially means that it is induced from its
restriction to~$\Or(G,\cF)$, up to homotopy, \ie it belongs to the
image of the functor on homotopy categories
$$
\LInd:\Ho\big(\Fun(\Or(G,\cF),\cM)\big)\too
\Ho\big(\Fun(\Or(G),\cM)\big)
$$
left adjoint to the obvious functor in the other direction, defined
by restriction from $\Or(G)$ to $\Or(G,\cF)$. This is explained
in~\cite{BM}, where we say that the functor \emph{$\bfE$ satisfies
$\Or(G,\cF)$-codescent} if $\bfE$ belongs to the image of $\LInd$ up
to isomorphism in $\Ho\big(\Fun(\Or(G),\cM)\big)$. This is
equivalent to the $\cF$-assembly property for~$G$ and for all its
subgroups. However, we shall not use the language of~\cite{BM} here.

\subsection{Transport groupoid}\label{sub:transport}
Let $S$ be a left $G$-set. The {\em transport groupoid} $\overline{S}$ associated to $S$ has $S$ as the set of objects and the following morphisms
$$ \Hom_{\overline{S}}(s,t):=\{g \in G\,|\, gs=t\}$$
for $s, t \in S$.
Composition is given by group multiplication. This defines a
functor
$$\overline{?}: \Or(G) \too \mathsf{Grp}$$
from the orbit category to the category of groupoids. Note that for every subgroup $H$ of $G$, the groupoid $\overline{G/H}$ is
connected. Hence it is equivalent to the full subcategory on any of
its objects, for instance the canonical object~$eH\in G/H$, whose group of automorphisms is~$H$. So, if we think of
the group $H$ as a one-object category, denoted $\underline{H}$, we have
an equivalence of groupoids $\underline{H}\isoto \overline{G/H}$. In
other words, the groupoid $\overline{G/H}$ is a natural
several-object replacement of the group~$H$.

\subsection{$R$-linearization}\label{sub:linearization}
We now recall the passage from groupoids to $R$-categories, \ie
additive categories enriched over the symmetric monoidal category of $R$-modules. Let $\cC$ be a groupoid. The {\em associated $R$-category $R[\cC]$} is the idempotent completion of the $R$-category $R[\cC]_\oplus$ whose objects are the formal finite direct sums of objects of~$\cC$ and
whose morphisms are the obvious matrices with entries in the free
$R$-modules $R[\cC(X,Y)]$ generated by the sets $\cC(X,Y)$.
Composition in $R[\cC]_\oplus$ is induced from composition in~$\cC$ and matrix multiplication. Idempotent completion is the usual formal creation of images and kernels for idempotent endomorphisms. The construction $\cC \mapsto R[\cC]$ yields a well-defined functor
$$ R[-]: \mathsf{Grp} \too R\text{-}\cat$$
with values in the category of (idempotent complete) small $R$-categories. For instance, for a one-object groupoid $\underline{H}$, the category $R[\underline{H}]_\oplus$ is equivalent to that of free $RH$-modules of finite rank and its idempotent completion $R[\underline{H}]$ is equivalent to the category of finitely generated projective $RH$-modules.

\subsection{$K$-theory}
Recall from~\cite{Pedersen-Weibel} that we can associate to every $R$-category $\cC$ its non-connective $K$-theory spectrum $\bbK(\cC)$, defining a functor $\bbK: R\text{-}\cat \to \spe$. Putting all these constructions together, we obtain the following composed functor
\begin{equation}\label{eq:K-functor}
\Or(G) \stackrel{\overline{?}}{\too} \mathsf{Grp} \stackrel{R[-]}{\too}  R\text{-}\cat \stackrel{\bbK}{\too} \spe\,.
\end{equation}
As usual, one obtains the $K$-theory \emph{groups}
$\K_\ast$ by taking (stable) homotopy groups. Thanks to the arguments in \S\,\ref{sub:transport}-\ref{sub:linearization} we have the following identifications
$$\K_{\ast}(RH)
= \pi_{\ast} \,\bbK(RH) \cong \pi_{\ast}
\,\bbK(R[\underline{H}]) \cong \pi_{\ast}
\,\bbK(R[\overline{G/H}]),
$$
which explain why the $K$-theory functor \eqref{eq:K-functor} defined on $\Or(G)$ is indeed the
expected one. This allowed Davis and L\"uck to prove in~\cite{DL} the equivalence between the Farrell-Jones conjecture in $K$-theory for $G$ and the $(\bbK, \cV C, G)$-isomorphism conjecture, \ie the statement that the functor \eqref{eq:K-functor} has the $\cV C$-assembly property.

Of course, there is also a classical connective $K$-theory functor, here simply denoted by $K:R\text{-}\cat \to \spe$. We now discuss a connection between $\bbK$ and~$K$.

\begin{proposition}[{L\"uck-Reich~\cite[Prop.~70]{Handbook}}]\label{prop:reduction}
Let $R$ be a regular ring in which the orders of all finite subgroups of $G$ are invertible. The $(\bbK,\cV C,G)$-isomorphism conjecture (\ie the Farrell-Jones conjecture) implies the $(K, \cF in, G)$-isomorphism conjecture, and a fortiori the $(K, \cV C, G)$-isomorphism conjecture.
\end{proposition}
\begin{proof}
We have a commutative diagram of natural maps
$$
\xymatrix@C=4em@R=1.5em{
\underset{\Or(G,\cF in)}\hocolim\, K {\vphantom{{{I^I}^I}^I}} \ \ar[d]_-\gamma^-{\simeq} \ar[rr]^-\alpha
&& \ K(RG) \ar[d]^-\beta
\\
\ \underset{\Or(G,\cF in)}\hocolim\, \bbK {\vphantom{{{I^I}^I}^I}} \ \ar[r]^-\delta_-{\simeq}
& \ \underset{\Or(G,\cV C)}\hocolim\, \bbK {\vphantom{{{I^I}^I}^I}} \ \ar[r]^-\epsilon
& \ \bbK(RG){\vphantom{{{I}^I_j}^I_j}}\,.
}$$
Under the stated assumptions on $R$, \cite[Prop.~70]{Handbook} implies that $\delta$ is a stable weak equivalence.
Moreover, as shown in the proof of~\cite[Prop.~70]{Handbook}, the group rings $RH$, with $H<G$ finite, are regular rings. This implies that $\gamma$ is a stable weak equivalence. Since $\beta$ induces a monomorphism on stable homotopy groups, if $\epsilon$ is a stable weak equivalence then so is $\alpha$.
\end{proof}

\begin{remark}\label{rk:converse}
Conversely, under the above assumptions about $R$ and $G$, one expects the spectrum $\bbK(RG)$ to be connective; see \cite[\S\,2.4.1]{Handbook}. If this is the case, the above proof also gives the converse to the statement of Proposition~\ref{prop:reduction}.
\end{remark}

\section{Dg categories}\label{se:dgcat}
We review some aspects of the theory of dg categories and introduce the notion of strictly finite dg cell. For a survey article, we invite the reader to consult Keller~\cite{ICM}.

Let $\cA$ be a small dg category (\S\,\ref{sub:NCM}). The {\em
opposite dg category} $\mathcal{A}\op$ of $\cA$ has the same objects
as $\mathcal{A}$ and complexes of morphisms given by
$\mathcal{A}\op(x,y):=\mathcal{A}(y,x)$. The category $\dgZo(\cA)$ has the same objects as $\cA$ and morphisms given by $\dgZo(\cA)(x,y):=\textrm{Z}^0(\cA(x,y))$, the 0-cocycles in the cochain complex $\cA(x,y)$. The {\em homotopy
category} $\dgHo(\cA)$ of $\cA$ has the same objects as $\cA$ and
morphisms given by $\dgHo(\cA)(x,y):=
\textrm{H}^0(\cA(x,y))$. Recall from~\cite[\S\,3.1]{ICM} that a {\em
right dg $\cA$-module} (or simply an $\cA$-module) is a dg functor $\cA\op \rightarrow
\cC_{\dg}(R)$, with values in the dg category $\cC_{\dg}(R)$ of
complexes of $R$-modules. We denote by $\cC(\cA)$ (resp.\ by
$\cC_{\dg}(\cA)$) the category (resp.\ dg category) of $\cA$-modules. Recall from~\cite[Thm.~3.2]{ICM} that $\cC(\cA)$ carries a standard projective model structure. The {\em derived category $\cD(\cA)$ of $\cA$} is the localization of $\cC(\cA)$ with respect to quasi-isomorphisms. Finally, let $\perf(\cA)$ be the dg category of {\em perfect} $\cA$-modules, \ie the
full dg subcategory of $\cC_{\dg}(\cA)$ spanned by the cofibrant $\cA$-modules that become compact~\cite[Def.~4.2.7]{Neeman} in the triangulated category $\cD(\cA)$.
\subsection{Strictly finite dg cells}\label{sse:dg}
Let $\underline{R}$ be the small dg category with one
object $\ast$ and such that $\underline{R}(\ast,\ast):=R$ (in degree
zero), where $R$ is the base ring. For $n \in \mathbb{Z}$, let $S^{n}$ be the complex $R[n]$
(with $R$ concentrated in degree $n$) and let $D^n$ be the mapping
cone on the identity of $S^{n-1}$. We denote by $\dgS(n)$ the dg
category with two objects $1$ and $2$ such that $ \dgS(n)(1,1)=R \ko
\dgS(n)(2,2)=R \ko \dgS(n)(2,1)=0  \ko \dgS(n)(1,2)=S^{n} $ and
composition given by multiplication. We denote by $\dgD(n)$ the dg
category with two objects $3$ and $4$ such that $ \dgD(n)(3,3)=R \ko
\dgD(n)(4,4)=R \ko \dgD(n)(4,3)=0 \ko \dgD(n)(3,4)=D^n $ and with
composition given by multiplication. Finally, let
$\iota(n):\dgS(n-1)\to \dgD(n)$ be the dg functor that sends $1$ to
$3$, $2$ to $4$ and $S^{n-1}$ into $D^n$ via the map $\incl:S^{n-1}\to D^{n}$ which is the identity on $R$ in
degree~$n-1$\,:
$$
\vcenter{
\xymatrix@C=.5em@R=1em{
\dgS(n-1) \ar@{=}[d]
&\ar[rrr]^{\displaystyle \iota(n)}&&&& \dgD(n) \ar@{=}[d]
\\
&&&&&\\
\\
1 \ar@(ul,ur)[]^{R} \ar[dd]^-{S^{n-1}}
& \kern1em \ar@{|->}[rrr]
&&&\kern1em
& 3\ar@(ul,ur)[]^{R}  \ar[dd]_-{D^n}
\\
& \kern1em \ar[rrr]^-{\incl} &&&\kern1em
\\
2 \ar@(dr,dl)[]^{R}
& \kern1em \ar@{|->}[rrr]
&&&\kern1em & 4\ar@(dr,dl)[]^{R}
}}
\qquad\text{where}\qquad
\vcenter{\xymatrix@R=1em@C=.8em{ S^{n-1} \ar[rr]^-{\incl} \ar@{=}[d]
&& D^n \ar@{=}[d]
\\
\ar@{.}[d]
&& \ar@{.}[d]
\\
0 \ar[rr] \ar[d]
&& 0 \ar[d]
\\
0 \ar[rr] \ar[d]
&& R \ar[d]^{\id}
\\
R \ar[rr]^{\id} \ar[d]
&& R \ar[d]
&{\scriptstyle(\textrm{degree }n-1)}
\\
0 \ar[rr] \ar@{.}[d]
&& 0 \ar@{.}[d]
\\
&&}}
$$
We denote by $I$ the set consisting of the dg functors
$\{\iota(n)\}_{n\in \mathbb{Z}}$ and the dg functor $\emptyset
\rightarrow \underline{R}$ (where the empty dg category $\emptyset$
is the initial one).

\begin{definition}\label{I-cell}
A small dg category $\cA$ is a {\em \dgcell} (compare with {Hirschhorn}~\cite[Def.~10.5.8]{Hirschhorn}) if
it is obtained from $\emptyset$ by a finite number of pushouts along
the dg functors of the set $I$. We denote by $\dgcatsf$ the full subcategory of $\dgcat$ consisting of \dgcells.
\end{definition}

\subsection{Quillen model structure}\label{sub:Morita}
Recall from~\cite[Thm.~5.3]{IMRN} that the category $\dgcat$ is
endowed with a (cofibrantly generated) {\em derived Morita} model structure,
whose weak equivalences are the {\em derived Morita dg functors}, \ie the dg
functors $F:\cA \rightarrow \cB$ which induce an equivalence on the
derived categories $\cD(\cB) \stackrel{\sim}{\rightarrow} \cD(\cA)$.
We denote by $\Ho(\dgcat)$ the homotopy category
hence obtained.

\subsection{Internal Hom-functor}\label{sub:internal}
Given dg categories $\cB$ and $\cA$ their {\em tensor product}
$\cB \otimes \cA$ is defined as follows. The set of objects is the cartesian product and, given objects $(z,x)$ and $(w,y)$ in $\cB \otimes \cA$, we set $(\cB \otimes \cA)((z,x),(w,y)):= \cB(z,w) \otimes \cA(x,y)$. This tensor product can be naturally derived into a bifunctor
\begin{equation}\label{eq:bifunctor}
-\otimes^{\Left}-: \Ho(\dgcat) \times \Ho(\dgcat) \too \Ho(\dgcat)\,,
\end{equation}
which gives rise to a symmetric monoidal structure on $\Ho(\dgcat)$. By To\"en~\cite[Thm.~6.1]{Toen} the bifunctor \eqref{eq:bifunctor} admits an internal Hom-functor $\rep(-,-)$.\footnote{Denoted by $\bbR\uHom(-,-)$ in {\em loc. cit.}} Given small dg categories $\cB$ and $\cA$, $\rep(\cB,\cA)$ is the full dg subcategory of $\cC_\dg(\cB\op\otimes^{\bbL}\cA)$ spanned by the cofibrant $\cB\text{-}\cA$-bimodules $X$ such that, for every object $z$ in $\cB$, the $\cA$-module $X(z,-)$ belongs to $\perf(\cA)$; by a $\cB\text{-}\cA$-bimodule we mean a dg functor $\cB\op\otimes \cA \to \cC_\dg(R)$, \ie a $\cB\op\otimes \cA$-module. Equivalently, $\rep(\cB,\cA)$ is formed by the cofibrant $\cB\text{-}\cA$-bimodules $X$ such that the induced functor
$$ - \otimes^{\bbL}_{\cB} X : \cD(\cB) \too \cD(\cA)$$
takes the representable $\cB$-modules to perfect $\cA$-modules. Such a bimodule yields a functor $\dgHo(\cB) \to \dgHo(\perf(\cA))$, which suggests that $\rep(\cB,\cA)$ can be thought of as the dg category of representations ``up to homotopy'' of $\cB$ in perfect $\cA$-modules. Note that  $\rep(\underline{R},\cB)$ is derived Morita equivalent to $\perf(\cB)$ (see~\cite[\S 4]{ICM}).

\section{Additive invariants of dg categories}\label{se:add}
Recall from \S\,\ref{sub:NCM} the notion of additive invariant of dg categories\,; consult~\cite[\S\,15]{Additive} for further details. In this section we collect several examples of additive invariants and introduce a ``coefficients variant''.

\subsection{Connective $K$-theory}\label{sub:conn}
Given a small dg category $\cA$ the $R$-linear category $\dgZo(\perf(\cA))$ is a category with cofibrations and weak equivalences in the sense of Waldhausen~\cite{Wald}. The cofibrations are the morphisms of $\cA$-modules which admit retractions as morphisms of graded $\cA$-modules and the weak equivalences are the quasi-isomorphisms; see \cite[\S5.2]{ICM}. The {\em connective $K$-theory spectrum $K(\cA)$ of $\cA$} is obtained by applying Waldhausen's construction~\cite[\S\,1.3]{Wald} to $\dgZo(\perf(\cA))$. Thanks to \cite[Example~15.6]{Additive} this gives rise to an additive invariant of dg categories
$$K: \HO(\dgcat) \too \HO(\spe)\,.$$

\subsection{Non-connective $K$-theory}\label{sub:nonconn}
Given a small dg category $\cA$, its {\em non-connective $K$-theory spectrum $\bbK(\cA)$} is obtained by applying Schlichting's construction to the Frobenius pair naturally associated to $\dgZo(\perf(\cA))$\,; see~\cite[\S\,6.4]{Marco}. Thanks to~\cite[Thm.~10.9]{Additive} this gives rise to an additive invariant of dg categories
$$\bbK: \HO(\dgcat) \too \HO(\spe)\,.$$
\subsection{Homotopy $K$-theory}\label{sub:homot}
Recall from Weibel~\cite[\S\,1]{Weibel} the simplicial $R$-algebra $\Delta_{\bullet}$ (viewed as a simplicial object in $\dgcat$), where $\Delta_n:= R[t_0, \ldots, t_n]/\sum_{i=0}^n t_i -1$. Given a small dg category $\cA$, its {\em homotopy $K$-theory spectrum $KH(\cA)$} is given by
$\mathrm{hocolim}_n\, \bbK(\cA \otimes \Delta_n)$.
Then we obtain a well-defined morphism of derivators
$$KH: \HO(\dgcat) \too \HO(\spe)\,.$$
By construction it preserves filtered colimits and the terminal object. Since the functors $-\otimes \Delta_n$ send split exact sequences to split exact sequences, we conclude that $KH$ is also an example of an additive invariant.
\subsection{Hochschild and cyclic homology}\label{sse:HCH}
Let $\cA$ be a small dg category. Recall from \cite[\S~5.3]{ICM} the construction of the Hochschild and cyclic homology complexes $HH(\cA)$ and $HC(\cA)$, and of the mixed complex $C(\cA)$.
Thanks to~\cite[Thm.~10.7]{Additive}, this gives rise to additive invariants of dg categories
\begin{eqnarray*}
C: \HO(\dgcat) \too \HO(\Lambda\text{-}\Mod) && HH,HC: \HO(\dgcat) \too \HO(\Ch(R))\,,
\end{eqnarray*}
where $\Lambda:= R[B]/(B^2)$, with $B$ of degree $-1$ and $dB=0$, and $\Ch(R)$ denotes the category of complexes of $R$-modules endowed with its projective model structure.
\subsection{Topological Hochschild homology}\label{sub:THH}
Let $\cA$ be a small dg category. Recall from~\cite[\S\,3]{BlumbergMandell} or~\cite[\S\,8.1]{MacLane} the {\em topological Hochschild homology spectrum $THH(\cA)$}. Thanks to~\cite[Prop.~8.9]{MacLane} this gives rise to an additive invariant of dg categories
$$THH :\HO(\dgcat) \longrightarrow \HO(\spe)\,.$$
\begin{remark}
Any $R$-algebra $A$ can be seen as a small dg category $\underline{A}$ with one object and with $A$ as the dg algebra of endomorphisms concentrated in degree zero. Note that the above invariants \S\,\ref{sub:conn}-\ref{sub:THH} verify the ``agreement property'', \ie when we apply them to $\underline{A}$ we recover the classical invariants associated to~$A$.
\end{remark}
\subsection{Coefficients variant}\label{sub:coef}
Given a small dg category $\cB$, the functor $\rep(\cB,-)$ (see \S\,\ref{sub:internal}) naturally gives rise to a morphism of derivators
\begin{equation}\label{eq:morphD}
\rep(\cB,-): \HO(\dgcat) \too \HO(\dgcat)\,.
\end{equation}

\begin{lemma}\label{lem:coef}
If the dg category $\cB$ is homotopically finitely
presented (Def.\,\ref{def:h-fin}), then the morphism \eqref{eq:morphD} preserves filtered homotopy colimits, the terminal object, and split exact sequences.
\end{lemma}
\begin{proof}
The morphism \eqref{eq:morphD} clearly preserves the terminal object as well as split exact sequences. Since $\cB$ is homotopically finitely presented, \cite[Thm.~3.3(3)]{CT} (where $\rep$ was denoted by $\mathsf{rep}$) implies that
$\rep(\cB,-)$ also preserves filtered homotopy colimits.
\end{proof}
Let $E: \HO(\dgcat) \to \bbD$ be an additive invariant of dg categories and $\cB$ a homotopically finitely presented dg category. Thanks to Lemma~\ref{lem:coef} we can construct a new additive invariant $E(-; \cB):\HO(\dgcat)\to\bbD$ as follows
\begin{eqnarray*}
\cA & \mapsto & E(\cA;\cB):= E(\rep(\cB,\cA))\,.
\end{eqnarray*}
If $\cB=\underline{R}$, then the dg category $\rep(\underline{R}, \cA)$
is derived Morita equivalent to $\cA$ and so $E(-;\underline{R})$ reduces
to~$E$. Hence, when $\cB$ is a general homotopically finitely presented dg
category, $E(-;\cB)$ can be thought of as a ``coefficients variant'' of $E$.

\section{Reordering the model of the additive motivator}\label{se:Madd}
We modify the Quillen model for the additive motivator
$\Madd$ of dg categories. This will be the main technical tool in the proof of Theorem~\ref{thm:Imain}; see \S\,\ref{se:proofmain}.

\subsection{The original model}\label{sub:original}
In~\cite[\S\,5]{Additive} the second author introduced the small category
$\dgcatf$ of \emph{finite $I$-cells} as being the smallest full
subcategory of $\dgcat$ which contains the \dgcells\
(see \S\,\ref{sse:dg}) and which is stable under the co-simplicial and
fibrant resolution functors of~\cite[Def.~5.3]{Additive}. Then, he
considered the projective model structure on the category $\Fun(\dgcatf\op, \sSetp)$ of
presheaves of pointed simplicial sets and took its left Bousfield localization
\begin{equation}\label{eq:Bloc}
\Locuns \Fun(\dgcatf\op, \sSetp)
\end{equation}
with respect to sets of morphisms $\widetilde{\cE^s_{un}}$, $p$ and
$\Sigma$; see~\cite[\S\,14]{Additive} for details.
Heuristically, inverting
$\Sigma$ is responsible for inverting Morita
equivalences, inverting $p$ is responsible for preserving the terminal object and,
inverting $\widetilde{\cE^s_{un}}$ is responsible for mapping split exact
sequences of dg categories to split triangles in the homotopy
category.

\begin{remark}\label{rem:regul}
In $\Fun(\dgcatf\op, \sSetp)$, sequential homotopy colimits commute with finite products and homotopy pullbacks and so by Remark~\ref{rem:HO(M)}, the associated derivator is regular (Def.\,\ref{def:srp}). Since the domains and codomains of the sets of morphisms $\widetilde{\cE^s_{un}}$, $p$ and
$\Sigma$ are homotopically finitely presented (Def.\,\ref{def:h-fin}), Remark~\ref{rem:regular} implies that the derivator associated to the left Bousfield localization \eqref{eq:Bloc} is also regular.
\end{remark}
In~\cite[Def.~15.1]{Additive} the second author defined the additive motivator $\Madd$
as the triangulated derivator associated (as in~\ref{sse:deriv}) to
the stable model category of spectra of objects in \eqref{eq:Bloc}, \ie
\begin{equation}\label{eq:Madd}
\Madd:=\HO\Big(\Spe\big(\Locuns \Fun(\dgcatf\op, \sSetp)\big)\Big)\,.
\end{equation}
%

\subsection{A new Quillen model}\label{sec4.1}

Recall from Appendix~\ref{appendix:A} that since $\dgcatf$ is a small
category, the category $\Fun(\dgcatf\op, \spe)$ carries naturally a
simplicial projective model structure. Moreover, we have a natural (Quillen) identification
\begin{equation}\label{eq:identification}
\Spe(\Fun(\dgcatf\op, \sSet_{\bullet})) \simeq \Fun(\dgcatf\op, \spe)\,.
\end{equation}
Now, consider the Yoneda functor
\begin{eqnarray*}
h: \dgcatf  \too  \Fun(\dgcatf\op, \spe) &&  \cB  \mapsto  \Sigma^{\infty} \dgcatf(-,\cB)\,,
\end{eqnarray*}
where every set $\dgcatf(?,\cB)$ is considered as a simplicially-constant simplicial set and
$\Sigma^{\infty}(-)$ denotes the infinite suspension spectrum. If $F$ is a fibrant object in $\Fun(\dgcatf\op, \spe)$, we
have the following weak equivalences\,:
\begin{eqnarray*}
\Map(h(\cB),F) \simeq F(\cB)_0&&\uMap(h(\cB),F) \simeq F(\cB)\,;
\end{eqnarray*}
consult Remark~\ref{rk:presheaves} for the definition of $\Map$ and $\uMap$. We also have a {\em homotopical} Yoneda functor
\begin{eqnarray*}
\underline{h}:\dgcat  \too \Fun(\dgcatf\op, \spe)&& \cA  \mapsto  \Sigma^{\infty} \Map(-,\cA)\,,
\end{eqnarray*}
where $\Map(-,-)$ denotes the homotopy function complex (see App.\,\ref{appendix:A}) of the derived Morita model structure on $\dgcat$ (see \S\,\ref{sub:Morita}). By construction, homotopy (co)limits in
$\Fun(\dgcatf\op, \spe)$ are calculated objectwise. This implies
that the shift models in $\spe$ for the suspension and loop space functors in
$\Ho(\spe)$ (see Jardine~\cite[\S\,1]{Jardine1}) induce objectwise shift models in $\Fun(\dgcatf\op,
\spe)$ for the suspension and loop space functors in the
triangulated category $\Ho(\Fun(\dgcatf\op, \spe))$.

\begin{proposition}\label{newModel}
The additive motivator \eqref{eq:Madd} admits another Quillen
model
$$
\cMadd:= \Loc_{\Omega(\widetilde{\cE^s_{un}}),\, \Omega(p),\, \Omega(\Sigma)}
\Fun(\dgcatf\op, \spe)\,,
$$
where $\Omega(\widetilde{\cE^s_{un}})$, $\Omega(p)$ and
$\Omega(\Sigma)$ are obtained by stabilizing the sets
$\widetilde{\cE^s_{un}}$, $p$ and $\Sigma$ in $\Fun(\dgcatf\op,
\spe)$ under the objectwise loop space functor.
\end{proposition}
\begin{proof}
The proof follows from the combination of Theorem~\ref{commut}, \cite[Thms.~4.4 and~8.7]{Additive}, Remark~\ref{rem:regul} and the above identification \eqref{eq:identification}.
\end{proof}
Our new construction can be summed up as follows\,; compare with~\cite[Rem.~15.2]{Additive}.
$$
\xymatrix@R=1.5em{
\HO(\dgcatf) \ar[r] \ar[d]_{\HO(h)}
& \HO(\dgcat) \ar@/^1pc/[dl]_-{\Right\underline{h}}
\ar@/^2pc/[ddl]^{\Uadd} \\
\HO\big(\Loc_{\Omega(\Sigma)} \Fun(\dgcatf\op,\spe)\big)\ar[d]
\\
\Madd }
$$
Here, $\HO(\dgcatf)$ is the prederivator associated with the full
subcategory $\dgcatf$ of $\dgcat$ (see~\ref{eq:HO}
and~\cite[\S\,5]{Additive} for details) and $\Uadd$ is the
composition of the functor $\Right\underline{h}$ induced by Yoneda
and the localization morphism
$$ \HO\big(\Loc_{\Omega(\Sigma)} \Fun(\dgcatf\op,\spe)\big) \too \HO\big(\cMadd\big) =\Madd\,.$$
\begin{proposition}\label{prop:fib}
An object $F \in \cMadd$ is
fibrant if and only if the following four conditions are verified\,:
\begin{itemize}
\item[(1)] $F(\cB) \in \spe$ is stably fibrant, for all $\cB \in \dgcatf$.
\item[(2)] For every derived Morita equivalence $\cB \to \cB'$ in $\dgcatf$, the induced morphism $F(\cB') \to F(\cB)$ is a stable weak equivalence in $\spe$.
\item[(3)] $F(\emptyset) \in \spe$ is contractible.
\item[(4)]  Every (left-hand) split exact sequence in $\dgcatf$ (see~\cite[Def.~13.1]{Additive}) gives rise to a (right-hand) homotopy fiber sequence in $\Ho(\spe)$
$$
\begin{array}{lcccr}
\xymatrix{
\mathcal{B}' \ar[r]_I
& \mathcal{B} \ar@<-1ex>[l]_T \ar[r]_P &
\mathcal{B}'' \ar@<-1ex>[l]_S
}
& & \mapsto && F(\mathcal{B}'') \stackrel{F(P)}{\too} F(\mathcal{B}) \stackrel{F(I)}{\too}
F(\mathcal{B}')\,.
\end{array}
$$
\end{itemize}
\end{proposition}

\begin{proof}
Condition~(1) corresponds to the fact that $F$ is fibrant in
$\Fun(\dgcatf\op, \spe)$ since we use the projective model. Thanks to the shift models in $\Fun(\dgcatf\op,
\spe)$ for the suspension and loop space functors in
$\Ho(\Fun(\dgcatf\op, \spe))$, the construction of the localized
model structure yields\,: An object $F$ is $\Omega(\Sigma)$-local if and only if for every derived Morita equivalence $\cB \to \cB'$ in $\dgcatf$, the morphism $F(\cB') \rightarrow F(\cB)$ is a levelwise weak equivalence in $\sSet_{\bullet}$. Since $F(\cB')$ and $F(\cB)$ are stably fibrant this is equivalent to condition~(2). An object $F$ is $\Omega(p)$-local if and only if $F(\emptyset)_n$ is contractible for every $n \geq 0$. Since $F(\emptyset)$ is stably fibrant this is equivalent to condition~(3). We now discuss condition~(4). The construction of the set
$\Omega(\widetilde{\cE^s_{un}})$ (see~\cite[Not.~14.5]{Additive} and
Proposition~\ref{newModel}) and the fact that the functor
$$ \Map(?,F): \Ho(\Fun(\dgcatf\op, \spe))\op \longrightarrow \Ho(\sSet_{\bullet})$$
sends homotopy cofiber sequences into homotopy fiber sequences, implies that an object $F$ is $\Omega(\widetilde{\cE^s_{un}})$-local if and only if every split exact sequence in $\dgcatf$ induces a homotopy fiber sequence in $\Ho(\sSet_{\bullet})$ for every $n \geq 0$; see~\cite[Prop.~14.8]{Additive}.
$$
\begin{array}{lcccr}
\xymatrix{
\mathcal{B}' \ar[r]_I & \mathcal{B} \ar@<-1ex>[l]_T \ar[r]_P &
\mathcal{B}'' \ar@<-1ex>[l]_S
}
&& \mapsto && F(\mathcal{B}'')_n \stackrel{F(P)_n}{\too} F(\mathcal{B})_n \stackrel{F(I)_n}{\too}
F(\mathcal{B}')_n
\end{array}
$$
Once again, since $F(\cB')$, $F(\cB)$ and $F(\cB'')$ are stably fibrant, this is equivalent to condition~(4). The proof is then concluded, thanks to general Bousfield localization theory; see~\cite[Prop.~3.4.1]{Hirschhorn}.
\end{proof}
Using the description of the fibrant objects of Proposition~\ref{prop:fib}, we now prove a key technical result.
\begin{proposition}\label{homotcol}
For every small dg category $\cC \in \dgcatf$, the functor
$$ \uMap(h(\cC), -): \Ho(\cMadd) \longrightarrow \Ho(\spe)$$
preserves homotopy colimits.
\end{proposition}
\begin{proof}
We start by observing that by construction, the result holds in the
stable model category $\Fun(\dgcatf\op, \spe)$. Thanks to Remark~\ref{rem:HFS} it suffices to prove the following\,: If $\{F_j\}_{j \in J}$ is a diagram of fibrant objects in the localized category
$\cMadd$, then its homotopy colimit satisfies conditions~(2)-(4) of Proposition~\ref{prop:fib}. Since homotopy colimits in
$\Fun(\dgcatf\op, \spe)$ are calculated objectwise, conditions~(2)-(3) are clearly verified. In what concerns condition~(4), notice
that $\Ho(\spe)$ is a triangulated category and so the homotopy
fiber sequences
$$ F_j(\cB'') \longrightarrow F_j(\cB) \longrightarrow F_j(\cB')$$
are also homotopy cofiber sequences. This implies that
$$ \underset{j \in J}\hocolim\, F_j(\cB'') \longrightarrow \underset{j \in J}\hocolim\, F_j(\cB) \longrightarrow \underset{j \in J}\hocolim\, F_j(\cB')$$
is a homotopy cofiber sequence and so also a homotopy fiber sequence. This shows condition (4) and so the proof is finished.
\end{proof}
We finish this subsection by describing an explicit set of generators.
\begin{proposition}\label{GenDG}
The set of \dgcells\ $\{ h(\cB)\,|\,\cB \in \dgcatsf\}$ (Def.\,\ref{I-cell}) form a set of homotopic generators (Def.~\ref{def:HG}) in $\cMadd$.
\end{proposition}
\begin{proof}
Notice that the objects $\{ h(\cB)\,|\,\cB \in \dgcatf\}$ are homotopic generators in the model category $\Fun(\dgcatf\op, \spe)$ by the
very definition of weak equivalences. Recall from \S\,\ref{sub:original} that $\dgcatf$ is the smallest full subcategory of $\dgcat$ which contains the \dgcells\ and which is stable under the co-simplicial and fibrant resolution functors of~\cite[Def.~5.3]{Additive}. Therefore, every object in $\dgcatf$ is derived Morita
equivalent to an object in $\dgcatsf$. This implies by
Lemma~\ref{homgenerators} and Proposition~\ref{prop:fib}, that the
objects $\{ h(\cB)\,|\,\cB \in \dgcatsf\}$ are homotopic generators in
$\Loc_{\Omega(\Sigma)}\Fun(\dgcatf\op, \spe)$. Once again by
Lemma~\ref{homgenerators} we can localize further with respect
to the sets $\Omega(\widetilde{\cE^s_{un}})$ and $\Omega(p)$, which completes the proof.
\end{proof}

\section{Fundamental additive functor}\label{se:univ}
\setcounter{equation}{0}

We introduce the notion of additive functor on the orbit category, give several examples, and construct the fundamental functor which satisfies additivity.

Note that every $R$-category (see \S\,\ref{sub:linearization}) can be naturally considered as a dg
category (with complexes of morphisms concentrated in degree zero). Given a group~$G$, we thus obtain a composed functor
\begin{equation}\label{eq:composed}
\Or(G) \stackrel{\overline{?}}{\too} \mathsf{Grp} \stackrel{R[-]}{\too} R\text{-}\cat \subset \dgcat \,.
\end{equation}
This functor is the basic piece. We now consider all functors obtained from composing it with an additive invariant of dg categories.
\begin{definition}\label{addfunctor}
Let $\cM$ be a stable model category (see Rem.\,\ref{rem:HO(M)}) and $\bfE:
\Or(G) \to \cM$ a functor. We say that $\bfE$ is {\em additive} if
it factors through \eqref{eq:composed} followed by a functor $E:\dgcat \to \cM$ whose associated
morphism of derivators $E:\HO(\dgcat) \to \HO(\cM)$ is an additive
invariant of dg categories (see \S\,\ref{sub:NCM}).
\end{definition}
The factorization of Definition~\ref{addfunctor} should not be confused with the one we want to
establish in Theorem~\ref{thm:factor} (that is, via the fundamental
additive functor $\Efund$). We rather restrict attention to functors on the orbit
category that only depend on the associated dg category. This is a
mild restriction since many of the classical functors have been
extended to dg categories, as explained in \S\,\ref{se:add}.

\begin{examples}\label{ex:coef}
Recall from \S\,\ref{se:add} several examples of functors $E:
\dgcat \to \cM$ defined on the category of dg categories (\eg\ connective, non-connective, and homotopy $K$-theory, Hochschild and cyclic
homology, and topological Hochschild homology), whose associated morphisms of
derivators $E: \HO(\dgcat) \to \HO(\cM)$ are additive invariant
of dg categories. By pre-composing them with the functor \eqref{eq:composed}
we obtain several examples of additive functors $\bfE:\Or(G)\to\cM$ in
the sense of Definition~\ref{addfunctor}. Moreover, if $\cB$ is a
homotopically finitely presented dg category $\cB$, we obtain a ``coefficients variant'' $\bfE(-;\cB)$ (see \S\,\ref{sub:coef}) defined as follows
\begin{eqnarray*}
\Or(G) \ni G/H & \mapsto & \bfE(G/H;\cB):= E(\rep(\cB, R[\overline{G/H}]))\,.
\end{eqnarray*}
Note that if $\cB=\underline{R}$, the additive functor $\bfE(-;\cB)$ reduces to the composition
\begin{equation*}
\bfE: \Or(G) \stackrel{\overline{?}}{\too} \mathsf{Grp} \stackrel{R[-]}{\too} R\text{-}\cat \too \dgcat \stackrel{E}{\too} \cM\,.
\end{equation*}
\end{examples}
\begin{definition}
The {\em fundamental additive functor} $\Efund$ is the
composition
$$
\Or(G) \stackrel{\overline{?}}{\too} \mathsf{Grp}
\stackrel{R[-]}{\too} R\text{-}\cat \subset \dgcat
\stackrel{\Uadd}{\too} \cMadd\,.
$$
The universality theorem~\cite[Thm.~15.4]{Additive} (see equivalence~\eqref{eq:cat}) yields\,:
\end{definition}
\begin{theorem}\label{thm:factor}
Let $G$ be a group and $\bfE:\Or(G) \to \cM$ an additive functor. Then there exists a homotopy colimit preserving morphism of derivators $\overline{E}:\Madd \to \HO(\cM)$, which makes the following diagram commute (up to isomorphism)
\begin{equation*}
\xymatrix{
\Or(G) \ar[r]^{\Efund} \ar[dr]_{\bfE} & \cMadd \ar[r] & \Madd(e) \ar[d]^{\overline{E}(e)} \\
&\cM \ar[r] & \Ho(\cM)\,.
}
\end{equation*}
\end{theorem}
\begin{proof}
By Definition~\ref{addfunctor}, $\bfE$ factors through a functor
$E:\dgcat \to \cM$ whose associated morphism of derivators
$E:\HO(\dgcat) \to \HO(\cM)$ is an additive invariant of dg
categories. By~\cite[Thm.~15.4]{Additive}, see~\eqref{eq:cat}, this $E$ descends
to a homotopy colimit preserving morphism of derivators $\overline{E}:\Madd \to \HO(\cM)$,
whose value at the base category makes the above diagram commute (up to isomorphism).
\end{proof}
Using the general notion of assembly property of \S\,\ref{assemblage}, we get\,:
\begin{corollary}\label{cor:ass}
Let $G$ be a group and let $\cF$ be a family of subgroups. If the
fundamental additive functor $\Efund$ has the $\cF$-assembly
property, then so do all additive functors.
\end{corollary}

\begin{proof}
Simply apply the morphism $\overline{E}$ to the $(\Efund,\cF,G)$-assembly map and use the fact
that $\overline{E}$ preserves arbitrary homotopy colimits.
\end{proof}

\section{Reduction to \dgcells}\label{se:proofmain}

\begin{proof}[Proof of Theorem~\ref{thm:Imain}]
The main ingredient of the proof is the following result:
\begin{theorem}{(see\footnote{In {\em loc.}~{\em cit.} $\uMap$ was denoted by $\Hom^{\mathrm{Sp}^\mathbb{N}}$, $K$ by $K^c$, and $\rep$ by $\mathsf{rep}_{\mathrm{mor}}$.} \cite[Thm.~15.10]{Additive})}\label{thm:co-representability}
Given dg categories $\cA$ and $\cB$, with $\cB$ homotopically finitely presented, there is a canonical stable weak equivalence of spectra
$$\uMap_{\Madd}\left(\,\Uadd(\cB),\Uadd(\cA)[1]\,\right) \cong K(\rep(\cB,\cA))\,.$$
\end{theorem}
Thanks to Corollary~\ref{cor:ass}, condition~(1) of Thm.\,\ref{thm:Imain} implies condition~(2). Thanks to~\cite[Prop.~5.2 and Ex.~5.1]{Additive} a dg
category is homotopically finitely presented (Def.\,\ref{def:h-fin})
if and only if it is derived Morita equivalent to a retract in $\Ho(\dgcat)$ (see \S\,\ref{sub:Morita}) of a
\dgcell\ (Def.\,\ref{I-cell}). Therefore, every \dgcell\ is homotopically finitely presented, and so condition~(2) implies condition~(3).
We now show that condition~(3) implies condition~(1). Recall the construction of the fundamental additive functor
$$
\Efund: \Or(G) \stackrel{\overline{?}}{\too} \mathsf{Grp} \stackrel{R[-]}{\too} R\text{-}\cat \subset \dgcat \stackrel{\Uadd}{\too} \cMadd\,.
$$
Assuming condition~(3), we need to show that the induced map
$$
\underset{\Or(G, \cF)}\hocolim \,\, \Uadd(R[\overline{G/H}]) \too
\Uadd(R[\overline{G/G}])
$$
is an isomorphism in $\Madd(e)$. Since the category $\Madd(e)$ is
triangulated, it suffices to show that the suspension map
$$
\underset{\Or(G, \cF)}\hocolim \,\, \Big(
\Uadd(R[\overline{G/H}])[1]\Big) \simeq \Big(\underset{\Or(G, \cF)}\hocolim
\,\, \Uadd(R[\overline{G/H}])\Big)[1] \too \Uadd(R[\overline{G/G}])[1]
$$
is an isomorphism. By Proposition~\ref{GenDG}, the set of objects
$\{ h(\cB)\,|\,\cB \in \dgcatsf\}$ form a set of homotopic generators in $\cMadd$ and so it is enough to prove
that, for every $\cB \in \dgcatsf$, the induced map of spectra
$$
\uMap\big(h(\cB), \underset{\Or(G, \cF)}\hocolim \,\,
\Uadd(R[\overline{G/H}])[1]\big) \too
\uMap\big(h(\cB),\Uadd(R[\overline{G/G}])[1]\big)
$$
is a stable weak equivalence. By Proposition~\ref{homotcol}, the
functor $\uMap(h(\cB),-)$ preserves homotopy colimits and so we
have
$$
\uMap\big(h(\cB), \underset{\Or(G, \cF)}\hocolim \,\,
\Uadd(R[\overline{G/H}])[1]\big) \simeq \underset{\Or(G,
\cF)}\hocolim\,\, \uMap\big(h(\cB),
\Uadd(R[\overline{G/H}])[1]\big)\,.
$$
By construction of $\Madd$ (see \cite[\S15]{Additive}), we observe that since $\cB \in \dgcatsf$ the object $h(\cB)$ identifies with $\Uadd(\cB)$. Hence, the above co-representability theorem~\ref{thm:co-representability} provides stable weak equivalences
$$ \uMap(h(\cB), \Uadd(R[\overline{G/H}])[1]) \cong K(\rep(\cB, R[\overline{G/H}]))\,,$$
for every $\cB \in \dgcatsf$ and $H \in \Or(G,\cF)$. In conclusion,
we are reduced to show that for every \dgcell\ $\cB$, the map
$$\underset{\Or(G,\cF)}\hocolim \,\, K(\rep(\cB, R[\overline{G/H}])) \too
K(\rep(\cB, R[\overline{G/G}]))
$$
is a stable weak equivalence. But now, this is precisely our hypothesis,
namely that the additive functors $\KB{-}$ have the $\cF$-assembly
property for~$G$.
\end{proof}

\appendix
\section{Model category tools}\label{appendix:A}
\setcounter{equation}{0}

In this appendix we recall some material from the theory of Quillen model structures~\cite{Quillen} and prove a technical lemma concerning homotopic generators.

Let $\sSet$ (resp.\ $\sSetp$) be the model category of (pointed) simplicial sets ; see Goerss-Jardine~\cite[\S\, I]{GJ}.
Given a Quillen model category $\cM$, we denote by
$\Map(-,-):\cM\op\times\cM\to\Ho(\sSet)$ its homotopy function complex; see~\cite[Def.~17.4.1]{Hirschhorn}. Recall that if $\cM$ is a simplicial
model category~\cite[\S\, II.3]{GJ}, its homotopy function complex is
given, for $X, Y \in \cM$, by the simplical set $ \Map(X,Y)_n := \cM(X_c \otimes \Delta[n], Y_f)$, where $X_c$ is a cofibrant resolution of $X$ and $Y_f$ is a fibrant
resolution of $Y$. Moreover, if $\Ho(\cM)$ denotes the homotopy
category of $\cM$, we have an isomorphism
$\pi_0\Map(X,Y)\simeq \Ho(\cM)(X,Y)$.
\begin{definition}
\label{def:h-fin}
An object $X$ in $\cM$ is {\em homotopically finitely presented} if
for any diagram $Y:J\to \cM$ in $\cM$ (for any shape, \ie small
category, $J$), the induced map
$$\underset{j \in J}{\hocolim}\, \Map(X,Y_j) \too \Map(X,\underset{j \in J}{\hocolim}\, Y_j)$$
is an isomorphism in $\Ho(\sSet)$.
\end{definition}

Let $\spe$ be the (model) category of spectra~\cite[\S\, X.4]{GJ}. If
$X$ is a spectrum, we denote by $X[n],\, n \geq 0$ its {\em
$n^{\textrm{th}}$ suspension}, \ie the spectrum defined as $X[n]_m:=
X_{n+m},\,\, m \geq 0$. If $X$ and $Y$ are two spectra, we define
its {\em homotopy function spectrum} $\uMap(X,Y)$ by $\uMap(X,Y)_n:= \Map(X, Y[n])$, where the bonding maps are the natural ones.

\begin{remark}\label{rk:presheaves}
Let $I$ be a small category. By~\cite[Thm.~3.3]{Jardine1}, the
category of presheaves of spectra\footnote{In \cite{Jardine1} the category $\Fun(I\op,\spe)$ of presheaves of spectra is denoted by $\spe(I)$. We do not use that notation since it already appears in \eqref{eq:Madd} and \eqref{eq:identification} with a different meaning.} $\Fun(I\op,\spe)=\spe^{I\op}$ carries the {\em projective} model structure,
with weak-equivalences and fibrations defined objectwise. If we
denote by $\Map(-,-)$ its homotopy function complex, the {\em
homotopy function spectrum} between two presheaves $F$ and $G$ is given
(as in the case of spectra) by $\uMap(F,G)_n:= \Map(F, G[n])$, where $G[n]$ is the {\em $n^{\textrm{th}}$ objectwise suspension}
of~$G$.
\end{remark}
\begin{remark}\label{rem:HFS}
Let $S$ be a set of morphisms in $\Fun(I\op,\spe)$ and
$\Loc_S(\Fun(I\op,\spe))$ its left Bousfield
localization with respect to~$S$\,; see~\cite[Thm.~4.1.1]{Hirschhorn}. Since the categories $\Fun(I\op,\spe)$ and
$\Loc_S(\Fun(I\op,\spe))$ have the same cofibrations, and hence the same trivial fibrations, the simplicial cofibrant replacement functor $\Gamma^\ast$ (see \cite[\S16]{Hirschhorn}) is the same in both cases. Hence, the homotopy function spectrum of $\Loc_S(\Fun(I\op,\spe))$ can be computed as $\uMap(-,Q(-))$, where $Q(-)$ is a
fibrant resolution functor in $\Loc_S(\Fun(I\op,\spe))$.
\end{remark}
Let $\cM$ be a left Bousfield localization
of $\Fun(I\op,\spe)$.
\begin{definition}\label{def:HG}
A set of \emph{homotopic generators} is a set of objects $\{G_j\}_{j \in J}$ in $\cM$ such that a morphism $f:F \to F'$ is a weak equivalence in $\cM$ if (and only if) for every object $G_j$ the induced map of spectra
\begin{equation}\label{eq:induced}
f_{\ast}:  \uMap(G_j,F) \too \uMap(G_j,F')
\end{equation}
is a stable weak equivalence.
\end{definition}

\begin{lemma}\label{homgenerators}
Let $S$ be a set
of morphisms in~$\cM$. If the $\{G_j\}_{j \in J}$ are homotopic generators in $\cM$
then they are homotopic generators in $\Loc_S(\cM)$ as well.
\end{lemma}

\begin{proof}
We use the homotopy function spectrum $\uMap(-,Q(-))$ in $\Loc_S(\cM)$ as in Remark~\ref{rem:HFS}. Let $f:F \to F'$ be a morphism in
$\cM$ which induces a stable equivalences under $\uMap(G_j,Q(-))$
for all $j\in J$. Consider the commutative square
$$
\xymatrix{
F \ar[d]_{f} \ar[r]^-{\sim} & Q(F) \ar[d]^{Q(f)} \\
F' \ar[r]_-{\sim} & Q(F')\,.
}
$$
Since by hypothesis the $\{G_j\}_{j \in J}$ are homotopic generators in $\cM$, the map $Q(f)$ is a weak equivalence in $\cM$ and so a
weak equivalence in $\Loc_S(\cM)$. By the two-out-of-three property,
we conclude that $f$ is a weak equivalence in $\Loc_S(\cM)$.
\end{proof}

\section{Grothendieck Derivators: stabilization and localization}
\label{appendix}
In this appendix we give a brief introduction to derivators, recall some basic facts, and then prove that the
operations of stabilization (see~\cite[\S\,8]{Additive}) and left
Bousfield localization (see~\cite[\S\,4]{Additive}) commute.

\subsection{Derivators}\label{sse:deriv}
The original reference is Grothendieck's
manuscript~\cite{Grothendieck}. See also Maltsiniotis~\cite{Malt} or
a short account in Cisinski-Neeman~\cite[\S\,1]{CN}.

Derivators originate in the problem of higher homotopies in derived categories.
For a non-zero triangulated category $\mathcal{D}$ and for $X$ a
small category, it essentially never happens that the diagram
category $\Fun(X,\cD)=\cD^X$ remains triangulated (it already fails for
the category of arrows in~$\cD$, that is, for
$X=[1]=(\bullet\to\bullet)$).

Now, very often, our triangulated category $\cD$ appears as the
homotopy category $\mathcal{D}=\Ho(\cM)$ of some
model~$\cM$. In this case, we can consider the category
$\Fun(X,\cM)$ of diagrams in $\cM$, whose homotopy category
$\Ho(\Fun(X,\cM))$ is often triangulated and provides a reasonable
approximation for $\Fun(X,\cD)$. More importantly, one can let
$X$ move. This nebula of categories $\Ho(\Fun(X,\cM))$, indexed by
small categories~$X$, and the various functors and natural
transformations between them is what Grothendieck formalized into
the concept of \emph{derivator}.

A derivator $\bbD$ consists of a
strict contravariant $2$-functor from the $2$-category of small
categories to the $2$-category of all categories (a.\,k.\,a.\ a
prederivator)
$$
\bbD: \Cat\op \longrightarrow \CAT,
$$
subject to certain conditions. We shall not list
them here for it would be too long but we refer to~\cite[\S\,1]{CN}. The
essential example to keep in mind is the derivator $\bbD=\HO(\cM)$
associated to a cofibrantly generated Quillen model category~$\cM$
and defined for every small category~$X$ by
\begin{equation}\label{eq:HO}
\HO(\cM)\ (X)=\Ho\big(\Fun(X\op,\cM)\big)\,.
\end{equation}
We denote by $e$ the $1$-point category with one object and one
identity morphism. Heuristically, the category $\bbD(e)$ is the
basic ``derived" category under consideration in the
derivator~$\bbD$. For instance, if $\bbD=\HO(\cM)$ then
$\bbD(e)=\Ho(\cM)$.

\begin{definitions}\label{def:srp}
We now recall three slightly technical properties of derivators.
\begin{enumerate}
\item
A derivator $\bbD$ is {\em strong} if for every finite free category
$X$ and every small category $Y$, the natural functor $ \bbD(X
\times Y) \longrightarrow \Fun(X\op,\bbD(Y))$ (see \cite[\S\,1.10]{CN}) is full and
essentially surjective.
\item
A derivator $\bbD$ is {\em pointed} if for any closed immersion $i:Z
\rightarrow X$ in $\Cat$ the cohomological direct image functor
$i_{\ast}:\bbD(Z) \longrightarrow \bbD(X)$ has a right adjoint, and
if moreover and dually, for any open immersion $j:U \rightarrow X$
the homological direct image functor $ j_!: \bbD(U) \longrightarrow
\bbD(X)$ has a left adjoint; see details in~\cite[Def.~1.13]{CN}.
\item A derivator $\bbD$ is {\em triangulated} or {\em stable} if it is
pointed and if every global commutative square in $\bbD$ is
cartesian exactly when it is cocartesian; see details
in~\cite[Def.~1.15]{CN}.
\end{enumerate}
\end{definitions}

\begin{remark}\label{rem:HO(M)}
If $\cM$ is a cofibrantly generated Quillen model category, \cite[Prop.~2.15]{catder} applied to all the homotopy categories $\Ho(\Fun(X\op,\cM))$ allows us to conclude that the derivator
$\HO(\mathcal{M})$ is strong. If
$\cM$ is pointed then so is~$\HO(\cM)$. Finally, if $\cM$ is a stable model category, then its
associated derivator $\HO(\cM)$ is triangulated. In short, the reader who
wishes to restrict attention to derivators of the form $\HO(\cM)$
can as well consider properties (1)-(3) of
Definition~\ref{def:srp} as mild ones.
\end{remark}

\begin{theorem}[Maltsiniotis~\cite{Malt1}]
For any triangulated derivator $\bbD$ and small category $X$ the category $\bbD(X)$ has a canonical triangulated structure.
\end{theorem}
(An explicit description of the triangulated structure is also given~\cite[\S\,7.9]{CN}.)

\begin{notation}
Let $\bbD$ and $\bbD'$ be derivators. We denote by $\uHom(\bbD,\bbD')$ the
category of all morphisms of derivators and by $\HomC(\bbD,\bbD')$ the
category of morphisms of derivators which preserve homotopy
colimits\,; see details in Cisinski~\cite[\S\,3.25]{Cisinski}.
\end{notation}

\subsection{Stabilization}\label{sse:St}
In \cite{Heller,Heller-AMS} Heller developed the notion of a {\em small homotopy theory} or in other words the notion of a {\em regular} derivator, \ie a derivator $\bbD$ where sequential homotopy colimits commute with finite products and homotopy pullbacks; see \cite[IV, Section~5]{Heller-AMS}. For instance, if $\cM$ is a cofibrantly generated model category where sequential homotopy colimits commute with finite products and homotopy pullbacks, then the associated derivator $\HO(\cM)$ is regular. By adopting the latter notation, Heller's work~\cite{Heller} can de described as follows: let $\bbD$ be a regular pointed strong derivator. Heller constructed the universal morphism $\stab: \bbD \to \St(\bbD)$ towards a triangulated strong derivator, which preserves homotopy colimits, and satisfies the following universal property.
\begin{theorem}[Heller~\cite{Heller}]\label{hel}
Let $\bbT$ be a triangulated strong derivator. Then the morphism $\stab: \bbD \to \St(\bbD)$ induces an equivalence of categories
$$ (\stab)^{\ast}: \HomC(\St(\bbD),\bbT) \stackrel{\sim}{\longrightarrow} \HomC(\bbD,\bbT)\,.$$
\end{theorem}
\smallskip

\subsection{Left Bousfield localization}
Let $\bbD$ be a derivator and $S$ a class of
morphisms in the base category $\bbD(e)$.

\begin{definition}\label{defCis}
The derivator $\bbD$ admits a {\em left Bousfield localization} with
respect to~$S$ if there exists a morphism of derivators $ \gamma : \bbD \rightarrow \Loc_S\bbD$,
which preserves homotopy colimits, sends the elements of $S$ to
isomorphisms in $\Loc_S\bbD(e)$, and satisfies the following
universal property\,: For every derivator $\bbD'$ the morphism
$\gamma$ induces an equivalence of categories
$$\gamma^{\ast}: \HomC(\Loc_S\bbD,\bbD')
  \stackrel{\sim}{\too}
  \uHom_{\,!,S}(\bbD,\bbD')\,,$$
where $\uHom_{\,!,S}(\bbD,\bbD')$ denotes
the category of morphisms of derivators which preserve homotopy
colimits and send the elements of $S$ to isomorphisms in $\bbD'(e)$.
\end{definition}

\begin{remark}\label{rem:regular}
Let $\cM$ be a left proper, cellular model category and $\Loc_S\cM$
its left Bousfield localization (see~\cite[Thm.~4.1.1]{Hirschhorn}) with
respect to a set of morphisms~$S$. Then, the induced morphism of
derivators $\HO(\cM)\to \HO(\Loc_S\cM)$ is a left Bousfield
localization of derivators with respect to the image of $S$ in
$\Ho(\cM)$\,; see~\cite[Thm.~4.4]{Additive}. Moreover, if the domains and codomains of the set $S$ are homotopically finitely presented objects (Def.\,\ref{def:h-fin}), the functor $\Ho(\Loc_S\cM) \to \Ho(\cM)$, right adjoint to the localization functor, preserves filtered homotopy colimits\,; see the proof of~\cite[Lem.~7.1]{Additive}. Under these hypotheses, if $\HO(\cM)$ is regular then so is $\HO(\Loc_S\cM)$.
\end{remark}

\begin{remark}\label{rem:triang}
By~\cite[Lem.~4.3]{Additive}, the Bousfield localization $\Loc_S\bbD$
of a \emph{triangulated} derivator $\bbD$ remains triangulated as
long as $S$ is stable under the loop space functor. For more
general~$S$, to remain in the world of \emph{triangulated}
derivators, one has to localize with respect to the set $\Omega(S)$
generated by $S$ and loops, as follows.
\end{remark}

\begin{proposition}\label{prop:omega}
Let $\bbD$ be a triangulated derivator and $S$ a class of morphisms
in~$\bbD(e)$. Let us denote by $\Omega(S)$ the smallest class of
morphisms in $\bbD(e)$ which contains $S$ and is stable under the
loop space functor $\Omega:\bbD(e) \rightarrow \bbD(e)$. Then, for
any \emph{triangulated} derivator~$\bbT$, we have an equivalence of
categories
\begin{equation}\label{eq:omega}
\uHom_{\,!, \Omega(S)}(\bbD,\bbT) \simeq \uHom_{\,!,S}(\bbD,\bbT)\,.
\end{equation}
Heuristically, $\Loc_{\Omega(S)}\bbD$ is the \emph{triangulated}
localization of $\bbD$ with respect to~$S$.
\end{proposition}

\begin{proof}
For $F$ an element of $\uHom_{\,!}(\bbD,\bbT)$, the functor
$F(e):\bbD(e)\to\mathbb{T}(e)$ preserves homotopy colimits,
hence it commutes in particular with the suspension functor. Since
both $\bbD$ and $\bbT$ are triangulated, suspension and loop space
functors are inverse to each other. Hence $F(e)$ also commutes
with~$\Omega$. It is then obvious that $F(e)$ sends $S$ to
isomorphisms if and only if it does so with~$\Omega(S)$.
\end{proof}

\subsection{Commuting stabilization with localization}

Let $\bbD$ be a pointed, strong and regular derivator and $S$ a class of morphisms in $\bbD(e)$.
Assume that $\bbD$ admits a left Bousfield localization
$\Loc_{S}\bbD$ with respect to~$S$. We then obtain a derivator
$\Loc_{S}\bbD$ which is still pointed and strong. If it is also
regular (see Remark~\ref{rem:regular}), we can consider its stabilization $\St(\Loc_S\bbD)$ as
in \S\,\ref{sse:St}.

On the other hand, we can first consider the triangulated derivator $\St(\bbD)$. We still denote by $S$ the image of the class $S$ under the morphism of derivators $\stab: \bbD \to \St(\bbD)$. Suppose that the left Bousfield localization $\Loc_{\Omega(S)}\St(\bbD)$ by $\Omega(S)$ also exists. We then have two constructions
$$
\xymatrix@R=1em@C=1em{
&& \kern1em \bbD \kern1em \ar[lld]_-{\gamma} \ar[rrd]^-{\stab}
\\
\Loc_S\bbD \ar[rd]_-{\stab}
&&&& \St(\bbD) \ar[ld]^-{\gamma}
\\
&\St(\Loc_S\bbD) \kern-1em
&& \kern-1em \Loc_{\Omega(S)}\St(\bbD)
}
$$
and we claim that they agree, namely\,:

\begin{theorem}\label{commut}
With the above notations and hypotheses, the derivators
$\Loc_{\Omega(S)}\St(\bbD)$ and $\St(\Loc_S\bbD)$ are canonically
equivalent, under~$\bbD$.
\end{theorem}

\begin{proof}
Both derivators are triangulated (for $\Loc_{\Omega(S)}\St(\bbD)$,
see Remark~\ref{rem:triang}) and strong. So, it suffices to show
that for any triangulated strong derivator~$\bbT$, we have the
following equivalences of categories:
$$
\xymatrix@R=1em@C=.5em{
&& \uHom_{\,!,S}(\bbD,\bbT)
\\
\HomC(\Loc_S\bbD,\bbT)\ar[rru]_-{\simeq}^(.4){\gamma^{\ast}}\kern-2em
&&&\kern-4em\uHom_{\,!,S}(\St(\bbD),\bbT)\ar@{=}[r]^-{\textrm{\eqref{eq:omega}}} \ar[lu]^(.65){\simeq}_(.4){\stab^{\ast}}
& \uHom_{\,!,\Omega(S)}(\St(\bbD),\bbT)
\\
&\HomC(\St(\Loc_S\bbD),\bbT) \kern-2em \ar[lu]_(.4){\simeq}^(.6){\stab^{\ast}}
&& \kern-2em \HomC(\Loc_{\Omega(S)}\St(\bbD),\bbT)\ar[ur]^(.4){\simeq}_(.6){\gamma^{\ast}}\,.
}
$$
The two equivalences on the left-hand side as well as the
lower-right one all follow from Theorem~\ref{hel} or
Definition~\ref{defCis}. Equivalence $\stab^{\ast}:
\uHom_{\,!,S}(\St(\bbD),\bbT)\isotoo \uHom_{\,!,S}(\bbD,\bbT)$ requires
a comment\,: By Theorem~\ref{hel} we have an equivalence
$\stab^{\ast}:\uHom_{\,!}(\St(\bbD),\bbT)\isotoo \uHom_{\,!}(\bbD,\bbT)$
and it is straightforward to check that it preserves the above subcategories.
\end{proof}

\medbreak\noindent\textbf{Acknowledgments:} The authors are grateful to
Bernhard Keller and Bertrand To\"en for useful discussions, Denis-Charles Cisinski for pointing
out a missing hypothesis in Theorem~\ref{commut}, Arthur Bartels and Wolfgang L{\"u}ck for precious comments on a previous draft, and Holger Reich for kindly providing Proposition~\ref{prop:reduction}. The second author would like to thank the UCLA Mathematics Department for hospitality and excellent working conditions, when this work was initiated.

\end{document}